\documentclass{amsart}
\usepackage{amssymb}
\usepackage{amsmath}
\usepackage{amsfonts}

\newtheorem{theorem}{Theorem}
\theoremstyle{plain}

\newtheorem{definition}{Definition}

\newtheorem{lemma}{Lemma}

\newtheorem{proposition}{Proposition}
\newtheorem{remark}{Remark}

\numberwithin{equation}{section}
\input{tcilatex}

\begin{document}
\title[Weighted Stochastic Field Exponent Sobolev Spaces]{Weighted
Stochastic Field Exponent Sobolev Spaces and Nonlinear Degenerated Elliptic
Problem}
\author{Ismail AYDIN}
\address{Sinop University Faculty of Arts and Sciences Department of
Mathematics \\
Sinop, TURKEY}
\email{iaydin@sinop.edu.tr}
\urladdr{}
\thanks{}
\author{Cihan UNAL}
\address{Sinop University Faculty of Arts and Sciences Department of
Mathematics \\
Sinop, TURKEY}
\email{cunal@sinop.edu.tr}
\urladdr{}
\date{}
\dedicatory{}
\thanks{}

\begin{abstract}
In this study, we consider weighted stochastic field exponent function
spaces $L_{\vartheta }^{p(.,.)}\left( D\times \Omega \right) $ and $%
W_{\vartheta }^{k,p(.,.)}\left( D\times \Omega \right) $. Also, we
investigate some basic properties and embeddings of these spaces. Finally,
we present an application of these spaces to the stochastic partial
differential equations with stochastic field growth.
\end{abstract}

\maketitle

\section{Introduction}

Nonlinear partial differential equations arise in chemical and biological
problems, in formulating fundamental laws of nature, in different areas of
physics, applied mathematics and engineering (such as solid mechanics, fluid
dynamics, acoustics, nonlinear optics, plasma physics, quantum eld theory)
and numerous applications. To study these equations is a very difficult task
because there are no general methods to solve such equations. Moreover, the
existence and the uniqueness of the solutions are fundamental, hard-to-prove
questions for any nonlinear equation with given boundary conditions.

In various applications (such as elasticity, non-Newtonian fluids and
electrorheological fluids, see \cite{Ruz}), it can be seen the boundary
value obstacle problems for elliptic equations. Many of these type equations
have been investigated for constant exponents of nonlinearity but it seems
to be more realistic to assume the variable exponent.

Harjulehto et. al \cite{Har} investigate an overview of applications to
differential equations with non-standard growth. Also, Aoyama \cite{Aoy}
discussed the some properties of Lebesgue spaces with variable exponent on a
probability space. In 2014, Tian et. al \cite{Tian} introduced stochastic
field exponent function spaces $L^{p(.,.)}\left( D\times \Omega \right) $
and $W^{k,p(.,.)}\left( D\times \Omega \right) $. They gave also an
application to the stochastic partial differential equations with stochastic
field growth in these spaces. Moreover, Lahmi et. al \cite{Lah} proved the
existence of solutions for the nonlinear $p(.)$-degenerate problems
involving nonlinear operators. This study is a generalization of \cite{Lah}, 
\cite{Tian2} and \cite{Tian}. Stochastic partial differential equations have
many applications in finance, such as option pricing etc.

In this paper, we define weighted stochastic field exponent function spaces $%
L_{\vartheta }^{p(.,.)}\left( D\times \Omega \right) $ and $W_{\vartheta
}^{k,p(.,.)}\left( D\times \Omega \right) $, and discuss some basic
properties of these spaces. Finally, we discuss the existence and uniqueness
of the weak solution for the nonlinear degenerated weighted $p(.,.)$
elliptic problem (the stochastic partial differential equations with
stochastic field growth)%
\begin{equation}
\left\{ 
\begin{array}{cc}
-\func{div}A\left( x,t,u,\nabla u\right) +A_{0}\left( x,t,u,\nabla u\right)
=f\left( x,t\right) , & \left( x,t\right) \in D\times \Omega \\ 
u=0, & \left( x,t\right) \in \partial D\times \Omega%
\end{array}%
\right.  \label{1}
\end{equation}%
where $A\left( x,t,s,\xi \right) $ is a Carath\'{e}odory function, which is
measurable stochastic fields on $D\times \Omega $ and continuous for $s$ and 
$\xi $ under some conditions.

It is known that pseudo-monotone operators have been many applications in
nonlinear elliptic equations. For example, Browder \cite{Brow} studied a
class of pseudo-monotone operators and applied it to a kind of boundary
value problems for nonlinear elliptic equations.

By the theory of pseudo-monotone operators, our aim is to show the
compactness techniques and the existence of a least weak solution of (\ref{1}%
). The special case of the equation of (\ref{1}) is the following equation%
\begin{equation*}
\left\{ 
\begin{array}{cc}
-\func{div}\left( \vartheta (x,t)\left\vert \nabla u\right\vert
^{p(x,t)-2}\nabla u\right) +\vartheta (x,t)g(u)\left\vert \nabla
u\right\vert ^{p(x,t)-1}=f\left( x,t\right) , & \left( x,t\right) \in
D\times \Omega \\ 
u=0, & \left( x,t\right) \in \partial D\times \Omega%
\end{array}%
\right. .
\end{equation*}%
Let $\lambda $ be a product measure on $D\times \Omega $ and $u(x,t)$ be a
Lebesgue measurable stochastic field on $D\times \Omega $, where $D$ is a
bounded open subset of $%
\mathbb{R}
^{d}$ $\left( d>1\right) $, and $\left( \Omega ,\tciFourier ,P\right) $ is a
complete probability space.

\section{Weighted Stochastic Field Exponent Lebesgue and Sobolev Spaces}

\begin{definition}
We denote the family of all measurable functions $p\left( .,.\right)
:D\times \Omega \longrightarrow \left[ 1,\infty \right) $ (called a
stochastic field exponent). In this paper, the function $p\left( .,.\right) $
always denotes a stochastic field exponent. Moreover, we put 
\begin{equation*}
p^{-}=\underset{(x,t)\in D\times \Omega }{\text{essinf}}p(x,t)\text{, \ \ \
\ \ \ }p^{+}=\underset{(x,t)\in D\times \Omega }{\text{esssup}}p(x,t)\text{.}
\end{equation*}%
A positive, measurable and locally integrable function $\vartheta $ defined
on $D\times \Omega $ is called a weight function. Now, we introduce the
integrability conditions used on the framework of weighted variable Lebesgue
and Sobolev spaces 
\begin{eqnarray*}
(H_{1}) &:&\vartheta \in L_{loc}^{1}\left( D\times \Omega \right) \text{ and 
}\vartheta ^{-\frac{1}{p(.,.)-1}}\in L_{loc}^{1}\left( D\times \Omega \right)
\\
(H_{2}) &:&\vartheta ^{-s(.,.)}\in L^{1}\left( D\times \Omega \right) ,
\end{eqnarray*}%
where $s\left( .,.\right) $ is a positive function. The weighted modular
function $\rho _{p(.,.),\vartheta }$ on $D\times \Omega $ is defined by%
\begin{equation*}
\rho _{p(.,.),\vartheta }\left( u\right) =E\left( \dint\limits_{D}\left\vert
u\left( x,t\right) \right\vert ^{p\left( x,t\right) }\vartheta \left(
x,t\right) dx\right) =\dint\limits_{D\times \Omega }\left\vert u\left(
x,t\right) \right\vert ^{p\left( x,t\right) }\vartheta \left( x,t\right)
d\lambda ,
\end{equation*}%
where $d\lambda =d\lambda \left( x,t\right) =dxdt$.
\end{definition}

The spaces $L_{\vartheta }^{p(.,.)}\left( D\times \Omega \right) $ consist
of all measurable stochastic fields (functions) $u$ on $D\times \Omega $
such that $\dint\limits_{D\times \Omega }\left\vert u(x,t)\right\vert
^{p(x,t)}\vartheta (x,t)d\lambda <\infty $ and endowed with the Luxemburg
norm%
\begin{equation*}
\left\Vert u\right\Vert _{p(.,.),\vartheta }=\inf \left\{ \lambda >0:\rho
_{p(.,.),\vartheta }\left( \frac{u}{\lambda }\right) \leq 1\right\} \text{.}
\end{equation*}%
It is well known that $u\in L_{\vartheta }^{p(.,.)}\left( D\times \Omega
\right) $ if and only if $\left\Vert u\right\Vert _{p(.,.),\vartheta
}=\left\Vert u\vartheta ^{\frac{1}{p(.,.)}}\right\Vert _{p(.,.)}<\infty $.
Moreover, it is clear that, if the inequality $0<C\leq \vartheta $ is
satisfied, then $L_{\vartheta }^{p(.,.)}\left( D\times \Omega \right)
\hookrightarrow L^{p(.,.)}\left( D\times \Omega \right) .$ In this study, we
assume that $1<p^{-}\leq p(.,.)\leq p^{+}<\infty $ and $(H_{1}),(H_{2}).$
Moreover, we use the abbreviations and symbols; a.e., $\longrightarrow $ and 
$\rightharpoonup $ for almost everywhere, strong convergence and weak
convergence, respectively.

It can be seen that the space $L_{\vartheta }^{p(.,.)}\left( D\times \Omega
\right) $ is uniformly convex, so it is reflexive, see \cite{Dien4}.
Moreover, we denote by $L_{\vartheta ^{\ast }}^{q(.,.)}\left( D\times \Omega
\right) $ as the dual space of $L_{\vartheta }^{p(.,.)}\left( D\times \Omega
\right) $ where $\frac{1}{p(.,.)}+\frac{1}{q(.,.)}=1$ and $\vartheta ^{\ast
}=\vartheta ^{1-q(.,.)}.$

Now, we give the relationships between $\left\Vert .\right\Vert
_{p(.,.),\vartheta }$ and $\rho _{p(.,.),\vartheta }$ as follows.

\begin{proposition}
\label{proposition2}If $u\in L_{\vartheta }^{p(.,.)}\left( D\times \Omega
\right) $, then we have

\begin{enumerate}
\item[\textit{(i)}] $\left\Vert u\right\Vert _{p(.,.),\vartheta
}^{p^{-}}\leq \rho _{p(.,.),\vartheta }\left( u\right) \leq \left\Vert
u\right\Vert _{p(.,.),\vartheta }^{p^{+}}$ with $\left\Vert u\right\Vert
_{p(.,.),\vartheta }\geq 1$.

\item[\textit{(ii)}] $\left\Vert u\right\Vert _{p(.,.),\vartheta
}^{p^{+}}\leq \rho _{p(.,.),\vartheta }\left( u\right) \leq \left\Vert
u\right\Vert _{p(.,.),\vartheta }^{p^{-}}$ with $\left\Vert u\right\Vert
_{p(.,.),\vartheta }\leq 1.$
\end{enumerate}
\end{proposition}

\begin{theorem}
The inequality 
\begin{equation*}
E\left( \dint\limits_{D}\left\vert f(x,t)g(x,t)\right\vert dx\right) \leq
C\left\Vert f\right\Vert _{p(.,.),\vartheta }\left\Vert g\right\Vert
_{q(.,.),\vartheta ^{\ast }}
\end{equation*}%
holds for every $f\in L_{\vartheta }^{p(.,.)}\left( D\times \Omega \right) $
and $g\in L_{\vartheta ^{\ast }}^{q(.,.)}(D\times \Omega )$ with the
constant $C$ depends on $p(.,.)$ where $\frac{1}{p(.,.)}+\frac{1}{q(.,.)}=1$
and $\vartheta ^{\ast }=\vartheta ^{1-q(.)}$
\end{theorem}

\begin{proof}
If we consider the H\"{o}lder inequality, then we get 
\begin{eqnarray*}
E\left( \dint\limits_{D}\left\vert f(x,t)g(x,t)\right\vert dx\right) 
&=&E\left( \dint\limits_{D}\left\vert f(x,t)g(x,t)\right\vert \left(
\vartheta (x,t)\right) ^{\frac{1}{p\left( x,t\right) }-\frac{1}{p\left(
x,t\right) }}dx\right)  \\
&\leq &C\left\Vert f\vartheta ^{\frac{1}{p(.,.)}}\right\Vert
_{p(.,.)}\left\Vert g\vartheta ^{-\frac{1}{p(.,.)}}\right\Vert _{q(.,.)}
\end{eqnarray*}%
for some $C>0$. That is the desired result.
\end{proof}

\begin{theorem}
(see \cite{Tian2},\cite{Tian})\label{reflexive}The space $L_{\vartheta
}^{p(.,.)}\left( D\times \Omega \right) $ is a reflexive Banach space with
respect to norm $\left\Vert .\right\Vert _{p(.,.),\vartheta }$.
\end{theorem}

\begin{proposition}
The space $L_{\vartheta }^{p(.,.)}\left( D\times \Omega \right) $ is
continuously embedded in $L_{loc}^{1}\left( D\times \Omega \right) $. This
means that every function in $L_{\vartheta }^{p(.,.)}\left( D\times \Omega
\right) $ has distributional (weak) derivative.
\end{proposition}

\begin{proof}
Suppose that $u\in L_{\vartheta }^{p(.,.)}\left( D\times \Omega \right) $
and let $K=K_{1}\times K_{2}\subset D\times \Omega $ be a compact set. By
the H\"{o}lder inequality, there exists an $A_{K}>0$ such that 
\begin{eqnarray}
\rho _{p\left( .,.\right) ,K}\left( u\right) &=&E\left(
\dint\limits_{K_{1}}\left\vert u(x,t)\right\vert \vartheta ^{\frac{1}{p(x,t)}%
}\vartheta ^{-\frac{1}{p(x.t)}}dx\right)  \notag \\
&\leq &A_{K}\left\Vert u\vartheta ^{\frac{1}{p(.,.)}}\right\Vert
_{p(.,.),K}\left\Vert \vartheta ^{-\frac{1}{p(.,.)}}\right\Vert _{q(.,.),K}
\label{1yeni}
\end{eqnarray}%
where $\frac{1}{p(.,.)}+\frac{1}{q(.,.)}=1.$ It is obvious that $\left\Vert
\vartheta ^{-\frac{1}{p(.,.)}}\right\Vert _{q(.,.),K}<\infty $ if and only
if $\rho _{q(.,.),K}\left( \vartheta ^{-\frac{1}{p(.,.)}}\right) <\infty .$
Since $\vartheta ^{-\frac{1}{p(.,.)-1}}\in L_{loc}^{1}\left( D\times \Omega
\right) ,$ we have%
\begin{equation}
\rho _{q(.,.),K}(\vartheta ^{-\frac{1}{p(.,.)}})=E\left(
\dint\limits_{K_{1}}\vartheta ^{-\frac{q(x,t)}{p(x,t)}}dx\right)
=\dint\limits_{K}\vartheta ^{-\frac{1}{p(x,t)-1}}d\lambda =B_{K}<\infty .
\label{2}
\end{equation}%
If we use (\ref{1yeni}) and (\ref{2}), then the proof is completed.
\end{proof}

\begin{remark}
If $\vartheta ^{-\frac{1}{p(.,.)-1}}\notin L_{loc}^{1}\left( D\times \Omega
\right) $, then the space $L_{\vartheta }^{p(.,.)}\left( D\times \Omega
\right) $ might not be continuously embedded in $L_{loc}^{1}\left( D\times
\Omega \right) $.
\end{remark}

\begin{theorem}
\label{theorem7}Let $u\in L_{\vartheta }^{p(.,.)}\left( D\times \Omega
\right) $ and $u_{n}\in L_{\vartheta }^{p(.,.)}\left( D\times \Omega \right) 
$ with $\left\Vert u_{n}\right\Vert _{p(.,.),\vartheta }\leq C$ for some $%
C>0 $. If $u_{n}\longrightarrow u$ a.e. in $D\times \Omega $, then $%
u_{n}\rightharpoonup u$ in $L_{\vartheta }^{p(.,.)}\left( D\times \Omega
\right) $.
\end{theorem}

\begin{proof}
Since the space $L_{\vartheta }^{p(.,.)}\left( D\times \Omega \right) $ is
reflexive by Theorem \ref{reflexive}, we only need to show that \ 
\begin{equation*}
E\left( \dint\limits_{D}u_{n}gdx\right) \longrightarrow E\left(
\dint\limits_{D}ugdx\right)
\end{equation*}%
for each $g\in L_{\vartheta ^{\ast }}^{q(.,.)}(D\times \Omega )$ where $%
\frac{1}{p(.,.)}+\frac{1}{q(.,.)}=1$ and $\vartheta ^{\ast }=\vartheta
^{1-q(.,.)}.$ It is well known that $\left\Vert u_{n}\right\Vert
_{p(.,.),\vartheta }\leq C$ if and only if $\rho _{p(.,.),\vartheta }\left( 
\frac{u_{n}}{C}\right) \leq 1$ for each $n\in 
\mathbb{N}
$. This follows by the Fatou's Lemma and the definition of the norm that%
\begin{equation*}
E\left( \dint\limits_{D}\left\vert \frac{u}{C}\right\vert ^{p(x,t)}\vartheta
(x,t)dx\right) \leq \underset{n\longrightarrow \infty }{\lim \inf }E\left(
\dint\limits_{D}\left\vert \frac{u_{n}}{C}\right\vert ^{p(x,t)}\vartheta
(x,t)dx\right) \leq 1.
\end{equation*}%
Thus, we get $\left\Vert u\right\Vert _{p(.,.),\vartheta }\leq C$. By the
absolute continuity of the Lebesgue integral, we have%
\begin{equation*}
\lim_{meas(K)\longrightarrow 0}\dint\limits_{D\times \Omega }\left\vert
g\chi _{K}\right\vert ^{q(x,t)}\vartheta (x,t)d\lambda =0
\end{equation*}%
where $g\in L_{\vartheta ^{\ast }}^{q(.,.)}(D\times \Omega )$ and $K\subset
D\times \Omega $. This yields that $\underset{meas(K)\longrightarrow 0}{\lim 
}\left\Vert g\chi _{K}\right\Vert _{q(.,.),\vartheta ^{\ast }}=0$, and there
exists a $\delta >0$ such that 
\begin{equation}
\left\Vert g\chi _{K}\right\Vert _{q(.,.),\vartheta ^{\ast }}<\frac{%
\varepsilon }{4C}\left( 1+\frac{1}{p^{-}}-\frac{1}{p^{+}}\right) ^{-1}
\label{3}
\end{equation}%
for $meas(K)<\delta .$ By the Egorov theorem, there exists a set $L\subset
D\times \Omega $ such that $u_{n}\longrightarrow u$ uniformly on $L$ with $%
meas\left( \left( D\times \Omega \right) -L\right) <\delta $. If we choose $%
n_{0}$ such that $n\geq n_{0}$, then we have 
\begin{equation}
\max_{(x,t)\in L}\left\vert u_{n}-u\right\vert \left\Vert g\right\Vert
_{q(.,.),\vartheta ^{\ast }}\left\Vert \chi _{L}\right\Vert
_{p(.,.),\vartheta }\left( 1+\frac{1}{p^{-}}-\frac{1}{p^{+}}\right) <\frac{%
\varepsilon }{2}.  \label{4}
\end{equation}%
Let us denote $K=\left( D\times \Omega \right) -L.$ By (\ref{3}) and (\ref{4}%
), we have 
\begin{eqnarray*}
&&\left\vert E\left( \dint\limits_{D}u_{n}gdx\right) -E\left(
\dint\limits_{D}ugdx\right) \right\vert \\
&\leq &\tint\limits_{L}\left\vert u_{n}-u\right\vert \left\vert g\right\vert
d\lambda +\tint\limits_{K}\left\vert u_{n}-u\right\vert \left\vert
g\right\vert d\lambda \\
&\leq &\max_{(x,t)\in L}\left\vert u_{n}-u\right\vert E\left(
\dint\limits_{D}\left\vert g\chi _{L}\right\vert dx\right) +E\left(
\dint\limits_{D}\left\vert u_{n}-u\right\vert \left\vert g\chi
_{K}\right\vert dx\right) \\
&\leq &\max_{(x,t)\in L}\left\vert u_{n}-u\right\vert \left\Vert
g\right\Vert _{q(.,.),\vartheta ^{\ast }}\left\Vert \chi _{L}\right\Vert
_{p(.,.),\vartheta }\left( 1+\frac{1}{p^{-}}-\frac{1}{p^{+}}\right) \\
&&+\left\Vert u_{n}-u\right\Vert _{p(.,.),\vartheta }\left\Vert g\chi
_{K}\right\Vert _{q(.,.),\vartheta ^{\ast }}\left( 1+\frac{1}{p^{-}}-\frac{1%
}{p^{+}}\right) \\
&<&\varepsilon .
\end{eqnarray*}%
That is the desired result.
\end{proof}

\begin{definition}
We set the weighted stochastic field variable Sobolev spaces $W_{\vartheta
}^{k,p(.,.)}\left( D\times \Omega \right) $ by%
\begin{equation*}
W_{\vartheta }^{k,p(.,.)}\left( D\times \Omega \right) =\left\{ u\in
L_{\vartheta }^{p(.,.)}\left( D\times \Omega \right) :D^{\alpha }u\in
L_{\vartheta }^{p(.,.)}\left( D\times \Omega \right) ,0\leq \left\vert
\alpha \right\vert \leq k\right\}
\end{equation*}%
equipped with the norm 
\begin{equation*}
\left\Vert u\right\Vert _{W_{\vartheta }^{k,p(.,.)}\left( D\times \Omega
\right) }=\dsum\limits_{0\leq \left\vert \alpha \right\vert \leq
k}\left\Vert D^{\alpha }u\right\Vert _{p(.,.),\vartheta },
\end{equation*}%
where $\alpha \in 
\mathbb{N}
_{0}^{d}$ is a multi-index, $\left\vert \alpha \right\vert =\alpha
_{1}+\alpha _{2}+...+\alpha _{d}$ and $D^{\alpha }=\frac{\partial
^{\left\vert \alpha \right\vert }}{\partial _{x_{1}}^{\alpha
_{1}}...\partial _{x_{d}}^{\alpha _{d}}}$. The space $\left( W_{\vartheta
}^{k,p(.,.)}\left( D\times \Omega \right) ,\left\Vert .\right\Vert
_{W_{\vartheta }^{k,p(.,.)}\left( D\times \Omega \right) }\right) $ is a
reflexive Banach space by \cite[Theorem 2.5]{Tian}.

Moreover, the space $W_{\vartheta }^{1,p(.,.)}\left( D\times \Omega \right) $
is defined by 
\begin{equation*}
W_{\vartheta }^{1,p(.,.)}\left( D\times \Omega \right) =\left\{ u\in
L_{\vartheta }^{p(.,.)}\left( D\times \Omega \right) :\left\vert \nabla
u\right\vert \in L_{\vartheta }^{p(.,.)}\left( D\times \Omega \right)
\right\}
\end{equation*}%
with the norm $\left\Vert u\right\Vert _{W_{\vartheta }^{1,p(.,.)}\left(
D\times \Omega \right) }=\left\Vert u\right\Vert _{p(.,.),\vartheta
}+\left\Vert \nabla u\right\Vert _{p(.,.),\vartheta }$.

The space $W_{0,\vartheta }^{1,p(.,.)}\left( D\times \Omega \right) $ is the
closure of 
\begin{equation*}
C\left( D\times \Omega \right) =\left\{ u:u\left( .,t\right) \in
C_{0}^{\infty }\left( D\right) \text{ for each }t\in \Omega \right\}
\end{equation*}%
in $W_{\vartheta }^{1,p(.,.)}\left( D\times \Omega \right) $. Also, it is
obvious that $C\left( D\times \Omega \right) $ is a subspace of $%
W_{0,\vartheta }^{1,p(.,.)}\left( D\times \Omega \right) $, and the dual
space of $W_{0,\vartheta }^{1,p(.,.)}\left( D\times \Omega \right) $ is $%
W_{0,\vartheta ^{\ast }}^{-1,q(.,.)}\left( D\times \Omega \right) $, where $%
\frac{1}{p(.,.)}+\frac{1}{q(.,.)}=1$ and $\vartheta ^{\ast }=\vartheta
^{1-q(.,.)}$.

The following theorem present the Poincar\'{e} inequality for the weighted
Sobolev spaces $W_{0,\vartheta }^{k,p(.,.)}\left( D\times \Omega \right) $.
For the proof, we refer \cite{Unal1}.
\end{definition}

\begin{remark}
If we use the similar method in \cite[Theorem 7]{Unal1}, then we have the
Poincar\'{e} inequality for $W_{0,\vartheta }^{k,p(.,.)}\left( D\times
\Omega \right) ,$ that is, there exists a $C>0$ such that the inequality%
\begin{equation}
\left\Vert u\right\Vert _{p(.,.),\vartheta }\leq C\left\Vert \nabla
u\right\Vert _{p(.,.),\vartheta }  \label{5}
\end{equation}%
holds for every $u\in W_{0,\vartheta }^{k,p(.,.)}\left( D\times \Omega
\right) $ (or $u\in C\left( D\times \Omega \right) $).
\end{remark}

Therefore, the space $W_{0,\vartheta }^{1,p(.,.)}\left( D\times \Omega
\right) $ equipped with the norm 
\begin{equation*}
\left\Vert \left\vert u\right\vert \right\Vert _{W_{0,\vartheta
}^{1,p(.,.)}\left( D\times \Omega \right) }=\left\Vert \nabla u\right\Vert
_{p(.,.),\vartheta }
\end{equation*}%
for $u\in W_{0,\vartheta }^{1,p(.,.)}\left( D\times \Omega \right) .$ It is
note that the norms $\left\Vert .\right\Vert _{W_{\vartheta
}^{1,p(.,.)}\left( D\times \Omega \right) }$ and $\left\Vert \left\vert
.\right\vert \right\Vert _{W_{0,\vartheta }^{1,p(.,.)}\left( D\times \Omega
\right) }$ are equivalent on $W_{\vartheta }^{1,p(.,.)}\left( D\times \Omega
\right) .$ Then, $W_{\vartheta }^{1,p(.,.)}\left( D\times \Omega \right) $
is continuously embedded in $L_{\vartheta }^{p(.,.)}\left( D\times \Omega
\right) $ if and only if the inequality (\ref{5}) is satisfied for every $%
u\in W_{0,\vartheta }^{1,p(.,.)}\left( D\times \Omega \right) .$

\section{Compact Embedding Theorems}

In this section, we present several compact embeddings between the weighted
stochastic field variable Lebesgue and Sobolev spaces. Because, we need
these embeddings to investigate weak solutions of stochastic partial
differential equation (\ref{1}). Now, let us introduce the function $p^{\ast
}\left( .,.\right) $ and $p_{s}\left( .,.\right) $ defined by%
\begin{equation*}
p^{\ast }\left( x,t\right) =\left\{ 
\begin{array}{cc}
\frac{dp(x,t)}{d-p(x,t)} & \text{if }p\left( x,t\right) <d, \\ 
\infty & \text{if }p\left( x,t\right) \geq d,%
\end{array}%
\right. ,
\end{equation*}%
\begin{equation*}
p_{s}\left( x,t\right) =\frac{p(x,t)s(x,t)}{s(x,t)+1}<p(x,t),
\end{equation*}%
and we have 
\begin{equation*}
p_{s}^{\ast }\left( x,t\right) =\left\{ 
\begin{array}{cc}
\frac{dp_{s}(x,t)}{d\left( s(x,t)+1\right) -p(x,t)s(x,t)} & \text{if }%
p_{s}\left( x,t\right) <d, \\ 
\text{arbitrary} & \text{if }p_{s}\left( x,t\right) \geq d,%
\end{array}%
\right.
\end{equation*}%
for almost all $(x,t)\in D\times \Omega $.

\begin{proposition}
(see \cite{Mas})\label{proposition}Assume that the boundary of $D\times
\Omega $ possesses the cone property and $p\left( .,.\right) \in C\left( 
\overline{D\times \Omega }\right) $. If $q(.,.)\in C\left( \overline{D\times
\Omega }\right) $ and $1\leq q(x,t)\leq p^{\ast }\left( x,t\right) $ for $%
(x,t)\in \overline{D\times \Omega }$, then $W^{1,p(.,.)}\left( D\times
\Omega \right) $ is compactly embedded in $L^{q(.,.)}\left( D\times \Omega
\right) $.
\end{proposition}

\begin{theorem}
\label{theorem14}Assume that the boundary of $D\times \Omega $ possesses the
cone property, $p(.,.)\in C\left( \overline{D\times \Omega }\right) $ and $%
1<p(x,t)$ for all $(x,t)\in \overline{D\times \Omega }$. Suppose that

\begin{enumerate}
\item[\textit{(i)}] $0<\vartheta (x,t)\in L^{\alpha (.,.)}\left( D\times
\Omega \right) $ with $(x,t)\in D\times \Omega $, $\alpha (.,.)\in C\left( 
\overline{D\times \Omega }\right) $ and $1<\alpha ^{-}$.

\item[\textit{(ii)}] $1<q(x,t)<\frac{\alpha (x,t)-1}{\alpha (x,t)}p^{\ast
}\left( x,t\right) $ for all $(x,t)\in \overline{D\times \Omega }$.
\end{enumerate}

Then, there is a compact embedding from $W^{1,p(.,.)}\left( D\times \Omega
\right) $ to $L_{\vartheta }^{q(.,.)}(D\times \Omega ).$
\end{theorem}

\begin{proof}
For the proof, we use similar method in \cite[Theorem 2.1]{Fan}. Let $u\in
W^{1,p(.,.)}\left( D\times \Omega \right) $ and set $r(x,t)=\frac{\alpha
(x,t)}{\alpha (x,t)-1}q(x,t)=\alpha _{0}(x,t)q(x,t)$. Then \textit{(ii)}
implies $r(x,t)<p^{\ast }\left( x,t\right) .$ This follows that $%
W^{1,p(.,.)}\left( D\times \Omega \right) \hookrightarrow \hookrightarrow
L^{r(.,.)}\left( D\times \Omega \right) $ by Proposition \ref{proposition}.
Moreover, for $u\in W^{1,p(.,.)}\left( D\times \Omega \right) ,$ we get $%
\left\vert u\right\vert ^{q(x,t)}\in L^{\alpha _{0}(x,t)}\left( D\times
\Omega \right) $. By the H\"{o}lder inequality%
\begin{equation*}
E\left( \dint\limits_{D}\vartheta (x,t)\left\vert u\right\vert
^{q(x,t)}dx\right) \leq \left( 1+\frac{1}{p^{-}}-\frac{1}{p^{+}}\right)
\left\Vert \vartheta \right\Vert _{\alpha (.,.)}\left\Vert \left\vert
u\right\vert ^{q(.,.)}\right\Vert _{\alpha _{0}(.,.)}<\infty .
\end{equation*}%
Therefore, we have $W^{1,p(.,.)}\left( D\times \Omega \right) \subset
L_{\vartheta }^{q(.,.)}(D\times \Omega ).$ Now, let $\left( u_{n}\right)
_{n\in 
\mathbb{N}
}\subset W^{1,p(.,.)}\left( D\times \Omega \right) $ and $%
u_{n}\rightharpoonup 0$ in $W^{1,p(.,.)}\left( D\times \Omega \right) $.
Since $W^{1,p(.,.)}\left( D\times \Omega \right) $ is compactly embedded in $%
L^{r(.,.)}\left( D\times \Omega \right) $ by the Proposition \ref%
{proposition}, we get $u_{n}\longrightarrow 0$ in $L^{r(.,.)}\left( D\times
\Omega \right) $, that is, $\left\Vert u_{n}\right\Vert
_{r(.,.)}\longrightarrow 0.$ This yields $\rho
_{r(.,.)}(u_{n})\longrightarrow 0$ and 
\begin{equation*}
E\left( \dint\limits_{D}\left\vert u_{n}(x,t)\right\vert ^{r(x,t)}dx\right)
=\rho _{\alpha _{0}(.,.)}(\left\vert u_{n}(x,t)\right\vert
^{q(x,t)})\longrightarrow 0
\end{equation*}%
or equivalently%
\begin{equation*}
\left\Vert \left\vert u_{n}\right\vert ^{q(.,.)}\right\Vert _{\alpha
_{0}(.,.)}\longrightarrow 0.
\end{equation*}%
Thus, we have 
\begin{equation*}
E\left( \dint\limits_{D}\vartheta (x,t)\left\vert u_{n}\right\vert
^{q(x,t)}dx\right) \leq \left( 1+\frac{1}{p^{-}}-\frac{1}{p^{+}}\right)
\left\Vert \vartheta \right\Vert _{\alpha (.,.)}\left\Vert \left\vert
u_{n}\right\vert ^{q(.,.)}\right\Vert _{\alpha _{0}(.,.)}\longrightarrow 0
\end{equation*}%
which implies $\left\Vert u_{n}\right\Vert _{q(.,.),\vartheta
}\longrightarrow 0.$ That is the desired result.
\end{proof}

\begin{theorem}
\label{theorem12}Let all conditions in Proposition \ref{proposition} be
hold. Moreover, assume that the assumptions in Theorem \ref{theorem14}
replacing $p\left( .,.\right) $ by $p_{s}\left( .,.\right) $ are also
satisfied. Thus we obtain 
\begin{equation*}
W_{\vartheta }^{1,p(.,.)}\left( D\times \Omega \right) \hookrightarrow
\hookrightarrow L_{\vartheta }^{r(.,.)}\left( D\times \Omega \right) .
\end{equation*}%
where $r(x,t)\leq p_{s}^{\ast }\left( x,t\right) $ for all $\left(
x,t\right) \in D\times \Omega $ and $0<C\leq \vartheta $.
\end{theorem}

\begin{proof}
First of all, we will show that $W_{\vartheta }^{1,p(.,.)}\left( D\times
\Omega \right) $ is continuously embedded in $W^{1,p_{s}(.,.)}\left( D\times
\Omega \right) $. Let $u\in W_{\vartheta }^{1,p(.,.)}\left( D\times \Omega
\right) $. Then it is clear that $u,\left\vert \nabla u\right\vert \in
L_{\vartheta }^{p(.,.)}\left( D\times \Omega \right) $. If we consider the H%
\"{o}lder inequality, Proposition 2 and $\vartheta ^{-s(.,.)}\in L^{1}\left(
D\times \Omega \right) $, then we have%
\begin{eqnarray}
E\left( \dint\limits_{D}\left\vert \nabla u\right\vert
^{p_{s}(x,t)}dx\right) &=&E\left( \dint\limits_{D}\left\vert \nabla
u\right\vert ^{p_{s}(x,t)}\vartheta ^{\frac{p_{s}(x,t)}{p(x,t)}}\vartheta ^{-%
\frac{p_{s}(x,t)}{p(x,t)}}dx\right)  \notag \\
&\leq &C\left\Vert \left\vert \nabla u\right\vert ^{p_{s}(.,.)}\vartheta ^{%
\frac{p_{s}(.,.)}{p(.,.)}}\right\Vert _{\frac{p(.,.)}{p_{s}(.,.)}}\left\Vert
\vartheta ^{-\frac{s(.,.)}{s(.,.)+1}}\right\Vert _{s(.,.)+1}  \notag \\
&\leq &C\left( \rho _{p(.,.),\vartheta }\left( \left\vert \nabla
u\right\vert \right) \right) ^{\frac{1}{\gamma _{1}}}\left( \rho
_{s(.,.)}\left( \vartheta ^{-1}\right) \right) ^{\frac{1}{\gamma _{2}}} 
\notag \\
&\leq &CC_{1}\left( \rho _{p(.,.),\vartheta }\left( \left\vert \nabla
u\right\vert \right) \right) ^{\frac{1}{\gamma _{1}}}  \label{6}
\end{eqnarray}%
where%
\begin{equation*}
\gamma _{1}=\left\{ 
\begin{array}{cc}
\left( \frac{p}{p_{s}}\right) ^{-}, & \text{if }\left\Vert \left\vert \nabla
u\right\vert ^{p_{s}(.,.)}\vartheta ^{\frac{p_{s}(.,.)}{p(.,.)}}\right\Vert
_{\frac{p(.,.)}{p_{s}(.,.)}}\geq 1 \\ 
\left( \frac{p}{p_{s}}\right) ^{+}, & \text{if }\left\Vert \left\vert \nabla
u\right\vert ^{p_{s}(.,.)}\vartheta ^{\frac{p_{s}(.,.)}{p(.,.)}}\right\Vert
_{\frac{p(.,.)}{p_{s}(.,.)}}\leq 1.%
\end{array}%
\right.
\end{equation*}%
and%
\begin{equation*}
\gamma _{2}=\left\{ 
\begin{array}{cc}
s^{-}+1, & \text{if }\left\Vert \vartheta ^{-\frac{s(.,.)}{s(.,.)+1}%
}\right\Vert _{s\left( .,.\right) +1}\geq 1 \\ 
s^{+}+1, & \text{if }\left\Vert \vartheta ^{-\frac{s(.,.)}{s(.,.)+1}%
}\right\Vert _{s\left( .,.\right) +1}\leq 1.%
\end{array}%
\right.
\end{equation*}%
Hence, we get 
\begin{eqnarray}
\left\Vert \nabla u\right\Vert _{p_{s}(.,.)}^{\gamma _{3}} &\leq &E\left(
\dint\limits_{D}\left\vert \nabla u\right\vert ^{p_{s}(x,t)}dx\right) \leq
CC_{1}\left( \rho _{p(.,.),\vartheta }\left( \left\vert \nabla u\right\vert
\right) \right) ^{\frac{1}{\gamma _{1}}}  \notag \\
&\leq &CC_{1}\left\Vert \nabla u\right\Vert _{p(.,.),\vartheta }^{\frac{%
\gamma _{4}}{\gamma _{1}}}  \label{7}
\end{eqnarray}%
where%
\begin{equation*}
\gamma _{3}=\left\{ 
\begin{array}{cc}
p_{s}^{-}, & \text{if }\left\Vert \nabla u\right\Vert _{p_{s}(.,.)}\geq 1 \\ 
p_{s}^{+}, & \text{if }\left\Vert \nabla u\right\Vert _{p_{s}(.,.)}\leq 1%
\end{array}%
\right.
\end{equation*}%
and 
\begin{equation*}
\gamma _{4}=\left\{ 
\begin{array}{cc}
p^{+}, & \text{if }\left\Vert \nabla u\right\Vert _{p(.,.),\vartheta }\geq 1
\\ 
p^{-}, & \text{if }\left\Vert \nabla u\right\Vert _{p(.,.),\vartheta }\leq 1%
\end{array}%
\right. .
\end{equation*}%
Therefore, we obtain 
\begin{equation}
\left\Vert \nabla u\right\Vert _{p_{s}(.,.)}\leq C^{\ast }\left\Vert \nabla
u\right\Vert _{p(.,.),\vartheta }^{\frac{\gamma _{4}}{\gamma _{1}\gamma _{3}}%
}.  \label{8}
\end{equation}%
Since $p_{s}\left( .,.\right) <p(.,.)$, we have $L_{\vartheta
}^{p(.,.)}\left( D\times \Omega \right) \hookrightarrow L^{p(.,.)}\left(
D\times \Omega \right) \hookrightarrow L^{p_{s}(.,.)}\left( D\times \Omega
\right) $, see \cite[Theorem 2.8]{Kor}. Then there exists $C^{\ast \ast }>0$
such that%
\begin{equation}
\left\Vert u\right\Vert _{p_{s}(.,.)}\leq C^{\ast \ast }\left\Vert
u\right\Vert _{p(.,.),\vartheta }  \label{9}
\end{equation}%
for almost everywhere in $D\times \Omega $. By (\ref{8}) and (\ref{9}), we
conclude that $W_{\vartheta }^{1,p(.,.)}\left( D\times \Omega \right)
\subset W^{1,p_{s}(.,.)}\left( D\times \Omega \right) $. If we consider the
Banach theorem in \cite{Car}, we get $W_{\vartheta }^{1,p(.,.)}\left(
D\times \Omega \right) \hookrightarrow W^{1,p_{s}(.,.)}\left( D\times \Omega
\right) $. This follows from Theorem \ref{theorem14} that 
\begin{equation*}
W^{1,p_{s}(.,.)}\left( D\times \Omega \right) \hookrightarrow
\hookrightarrow L_{\vartheta }^{r(.,.)}\left( D\times \Omega \right) .
\end{equation*}%
This completes the proof.
\end{proof}

Now, we reveal some required conditions for the equation (\ref{1}). Assume
that $A:%
\mathbb{R}
^{d}\times \Omega \times 
\mathbb{R}
\times 
\mathbb{R}
^{d}\longrightarrow 
\mathbb{R}
^{d}$ and $f:%
\mathbb{R}
^{d}\times \Omega \longrightarrow 
\mathbb{R}
$ satisfy the following growth conditions:

\begin{enumerate}
\item[$\left( H_{3}\right) $] $\left\vert A\left( x,t,s,\xi \right)
\right\vert \leq \beta \vartheta ^{\frac{1}{p(x,t)}}\left[ k(x,t)+\vartheta
^{\frac{1}{q(x,t)}}\left\vert \xi \right\vert ^{p(x,t)-1}\right] $

\item[$\left( H_{4}\right) $] $\left( A\left( x,t,s,\xi \right) -A\left(
x,t,s,\mu \right) \right) \left( \xi -\mu \right) >0$, $\xi \neq \mu $

\item[$\left( H_{5}\right) $] $A\left( x,t,s,\xi \right) \left\vert \xi
\right\vert \geq \alpha \vartheta (x,t)\left\vert \xi \right\vert ^{p(x,t)}$
\end{enumerate}

where $k(x,t)$ is a positive function in $L^{q(.,.)}(D\times \Omega )$ where 
$\frac{1}{p(.,.)}+\frac{1}{q(.,.)}=1$ and $\alpha $, $\beta $ are positive
constants.

Let $A_{0}\left( x,t,s,\xi \right) :%
\mathbb{R}
^{d}\times \Omega \times 
\mathbb{R}
\times 
\mathbb{R}
^{d}\longrightarrow 
\mathbb{R}
$ be a Carath\'{e}odory function such that for a.e. $\left( x,t\right) \in 
\mathbb{R}
^{d}\times \Omega $ and for all $s\in 
\mathbb{R}
$, $\xi \in 
\mathbb{R}
^{d}$, the growth condition%
\begin{equation}
\left\vert A_{0}\left( x,t,s,\xi \right) \right\vert \leq \gamma \left(
x,t\right) +g(s)\vartheta (x,t)\left\vert \xi \right\vert ^{p(x,t)-1}
\label{10}
\end{equation}%
is satisfied where $g:%
\mathbb{R}
\longrightarrow 
\mathbb{R}
^{+}$ is a continuous function that belongs to $L^{1}\left( 
\mathbb{R}
\right) $ and $\gamma \left( x,t\right) $ belongs to $L_{\vartheta ^{\ast
}}^{q\left( .,.\right) }\left( D\times \Omega \right) $. Finally, we assume
that $f\in W_{\vartheta ^{\ast }}^{-1,q(.,.)}\left( D\times \Omega \right) .$

\begin{lemma}
\label{lemma16}(see \cite{Lah})Assume that $\left( H_{3}\right) -\left(
H_{5}\right) $ hold and let $\left\{ u_{n}\right\} _{n\in 
\mathbb{N}
}$ be a sequence in $W_{0,\vartheta }^{1,p(.,.)}\left( D\times \Omega
\right) $ such that $u_{n}\rightharpoonup u$ in $W_{0,\vartheta
}^{1,p(.,.)}\left( D\times \Omega \right) $ and 
\begin{equation}
E\left( \dint\limits_{D}\left[ A\left( x,t,u_{n},\nabla u_{n}\right)
-A\left( x,t,u_{n},\nabla u\right) \right] \nabla \left( u_{n}-u\right)
dx\right) \longrightarrow 0\text{.}  \label{11}
\end{equation}%
Then $u_{n}\longrightarrow u$ in $W_{0,\vartheta }^{1,p(.,.)}\left( D\times
\Omega \right) $.
\end{lemma}

\section{Existence of Weak Solution of Stochastic Partial Differential
Equations With Stochastic Field Growth}

\begin{definition}
A function $u\in W_{0,\vartheta }^{1,p(.,.)}\left( D\times \Omega \right) $
is said to be a weak solution (\ref{1}), if 
\begin{equation}
E\left( \dint\limits_{D}\left[ A\left( x,t,u,\nabla u\right) \nabla \varphi
+A_{0}\left( x,t,u,\nabla u\right) \varphi \right] dx\right) =E\left(
\dint\limits_{D}f\left( x,t\right) \varphi dx\right)  \label{12}
\end{equation}%
for all $\varphi \in W_{0,\vartheta }^{1,p(.,.)}\left( D\times \Omega
\right) $.
\end{definition}

\begin{definition}
A bounded operator $T$ from $W_{0,\vartheta }^{1,p(.,.)}\left( D\times
\Omega \right) $ to its dual $W_{\vartheta ^{\ast }}^{-1,q(.,.)}\left(
D\times \Omega \right) $ is called pseudo-monotone if and only if for any
sequences $\left( u_{k}\right) _{k\in 
\mathbb{N}
}$ \text{in }$W_{0,\vartheta }^{1,p(.,.)}\left( D\times \Omega \right) $
satisfying

\begin{enumerate}
\item[\textit{(i)}] $u_{k}\rightharpoonup u$ in $W_{0,\vartheta
}^{1,p(.,.)}\left( D\times \Omega \right) $ as $k\longrightarrow \infty ,$

\item[\textit{(ii)}] $\limsup\limits_{k\longrightarrow \infty }\langle
T\left( u_{k}\right) ,u_{k}-u\rangle \leq 0$

imply $T\left( u_{k}\right) \rightharpoonup T\left( u\right) $ and $\langle
T\left( u_{k}\right) ,u_{k}\rangle \longrightarrow \langle T\left( u\right)
,u\rangle .$
\end{enumerate}
\end{definition}

\begin{definition}
Assume that $X$ is a reflexive Banach space and $X^{\ast }$ denotes dual of $%
X.$ Also, let $\langle .,.\rangle $ be a pair between $X$ and $X^{\ast }$.
Then a mapping $\Gamma :X\longrightarrow X^{\ast }$ is called coercive if
there exists a $u\in X$ such that 
\begin{equation*}
\frac{\langle \Gamma (u),u\rangle }{\left\Vert u\right\Vert _{X}}%
\longrightarrow \infty \text{ as }\left\Vert u\right\Vert
_{X}\longrightarrow \infty \text{.}
\end{equation*}
\end{definition}

Let us define the operator $\Gamma :W_{0,\vartheta }^{1,p(.,.)}\left(
D\times \Omega \right) \longrightarrow W_{\vartheta ^{\ast
}}^{-1,q(.,.)}\left( D\times \Omega \right) $ by 
\begin{equation*}
\langle \Gamma (u),\varphi \rangle =E\left( \dint\limits_{D}\left[ A\left(
x,t,u,\nabla u\right) \nabla \varphi +A_{0}\left( x,t,u,\nabla u\right)
\varphi \right] dx\right)
\end{equation*}%
where $\varphi \in W_{0,\vartheta }^{1,p(.,.)}\left( D\times \Omega \right) $
and $\frac{1}{p(.,.)}+\frac{1}{q(.,.)}=1$. Hence, we can write the equation (%
\ref{1}) as $\langle \Gamma (u),\varphi \rangle =\langle f,\varphi \rangle .$

\begin{proposition}
\label{proposition20}(Weak compactness of bounded set)Let $X$ be a reflexive
Banach space. Moreover, assume that $\left( u_{k}\right) _{k\in 
\mathbb{N}
}$ is a sequence such that

(i) $u_{k}\in X$

(ii) $\left\Vert u_{k}\right\Vert _{X}\leq C$ for all $k\in 
\mathbb{N}
,$

that is, $\left( u_{k}\right) _{k\in 
\mathbb{N}
}$ is a bounded sequence in $X,$ then there exists a subsequence $\left(
u_{k_{l}}\right) _{l\in 
\mathbb{N}
}$ and an element $u_{0}\in X$ such that $u_{k_{l}}\rightharpoonup u_{0}$ in 
$X.$
\end{proposition}

\begin{theorem}
\label{kritikteo}(see \cite{Ren})Let X be a reflexive Banach space and
assume $\Gamma :X\longrightarrow X^{\ast }$ is continuous (bounded),
coercive and pseudo-monotone. Then for every $g\in X^{\ast }$ there exists a
solution $u\in X$ of the equation $\Gamma \left( u\right) =g.$
\end{theorem}

Now, we are ready to give our main motivation of the paper.

\begin{theorem}
If the conditions $\left( H_{3}\right) ,$ $\left( H_{4}\right) $ and $\left(
H_{5}\right) $ hold, then there exists at least a weak solution of (\ref{1}).
\end{theorem}

\begin{proof}
The proof base on three parts.

\textbf{Step 1. }First of all, we will show that the operator $\Gamma $ is
bounded. The operator $\Gamma $ is equal to the sum of two operators such
that $\Gamma =\Gamma _{1}+\Gamma _{2}$ where 
\begin{equation*}
\langle \Gamma _{1}(u),\varphi \rangle =E\left( \dint\limits_{D}A\left(
x,t,u,\nabla u\right) \nabla \varphi dx\right)
\end{equation*}%
and%
\begin{equation*}
\langle \Gamma _{2}(u),\varphi \rangle =E\left( \dint\limits_{D}A_{0}\left(
x,t,u,\nabla u\right) \varphi dx\right) .
\end{equation*}%
If we consider $\left( H_{3}\right) $, Proposition \ref{proposition2} and H%
\"{o}lder inequality, then we have%
\begin{eqnarray*}
&&\left\vert \langle \Gamma _{1}(u),\varphi \rangle \right\vert \\
&=&\left\vert E\left( \dint\limits_{D}A\left( x,t,u,\nabla u\right)
\vartheta ^{-\frac{1}{p(x,t)}}\nabla \varphi \vartheta ^{\frac{1}{p(x,t)}%
}dx\right) \right\vert \\
&\leq &C\left\Vert A\left( x,t,u,\nabla u\right) \vartheta ^{-\frac{1}{p(x,t)%
}}\right\Vert _{q(.,.)}\left\Vert \nabla \varphi \vartheta ^{\frac{1}{p(x,t)}%
}\right\Vert _{p(.,.)} \\
&=&C\left\Vert A\left( x,t,u,\nabla u\right) \vartheta ^{-\frac{1}{p(x,t)}%
}\right\Vert _{q(.,.)}\left\Vert \left\vert \varphi \right\vert \right\Vert
_{W_{0,\vartheta }^{1,p(.,.)}\left( D\times \Omega \right) } \\
&\leq &C\left[ E\left( \dint\limits_{D}\left\vert A\left( x,t,u,\nabla
u\right) \right\vert ^{q(x,t)}\vartheta ^{-\frac{q(x,t)}{p(x,t)}}dx\right) %
\right] ^{\theta }\left\Vert \left\vert \varphi \right\vert \right\Vert
_{W_{0,\vartheta }^{1,p(.,.)}\left( D\times \Omega \right) } \\
&\leq &C\left[ E\left( \dint\limits_{D}\left\vert \beta \vartheta ^{\frac{1}{%
p(x,t)}}\left[ k(x,t)\right. \right. \right. \right. \\
&&\left. \left. \left. +\vartheta ^{\frac{1}{q(x,t)}}\left\vert \nabla
u\right\vert ^{p(x,t)-1}\right] \right\vert ^{q(x,t)}\left. \vartheta ^{-%
\frac{q(x,t)}{p(x,t)}}dx\right) \right] ^{\theta }\left\Vert \left\vert
\varphi \right\vert \right\Vert _{W_{0,\vartheta }^{1,p(.,.)}\left( D\times
\Omega \right) } \\
&\leq &C^{\ast }\left( \max \left\{ \beta ^{q^{-}},\beta ^{q^{+}}\right\}
\right) ^{\theta }\left[ E\left( \dint\limits_{D}\left[ \left\vert
k(x,t)\right\vert ^{q(x,t)}\right. \right. \right. \\
&&\left. \left. \left. +\left\vert \nabla u\right\vert ^{p(x,t)}\vartheta 
\right] dx\right) \right] ^{\theta }\left\Vert \left\vert \varphi
\right\vert \right\Vert _{W_{0,\vartheta }^{1,p(.,.)}\left( D\times \Omega
\right) } \\
&\leq &C^{\ast }\left( C_{1}+\rho _{p(.,.),\vartheta }(\nabla u)\right)
^{\theta }\left\Vert \left\vert \varphi \right\vert \right\Vert
_{W_{0,\vartheta }^{1,p(.,.)}\left( D\times \Omega \right) }
\end{eqnarray*}%
where%
\begin{equation*}
\theta =\left\{ 
\begin{array}{c}
\frac{1}{^{q^{-}}},\text{ if }\left\Vert A\left( x,t,u,\nabla u\right)
\right\Vert _{q(.,.),\vartheta ^{\ast }}\geq 1, \\ 
\frac{1}{^{q^{+}}},\text{ if }\left\Vert A\left( x,t,u,\nabla u\right)
\right\Vert _{q(.,.),\vartheta ^{\ast }}\leq 1%
\end{array}%
\right. ,
\end{equation*}%
and $\frac{1}{p(.,.)}+\frac{1}{q(.,.)}=1$. This yields that $\Gamma _{1}$ is
bounded. In similar way, since $\gamma \in L_{\vartheta ^{\ast }}^{q\left(
.,.\right) }\left( D\times \Omega \right) ,$ we get%
\begin{eqnarray*}
&&\left\vert \langle \Gamma _{2}(u),\varphi \rangle \right\vert \\
&=&\left\vert E\left( \dint\limits_{D}A_{0}\left( x,t,u,\nabla u\right)
\vartheta ^{-\frac{1}{p(x,t)}}\varphi \vartheta ^{\frac{1}{p(x,t)}}dx\right)
\right\vert \\
&\leq &C\left\Vert A_{0}\left( x,t,u,\nabla u\right) \vartheta ^{-\frac{1}{%
p(x,t)}}\right\Vert _{q(.,.)}\left\Vert \varphi \vartheta ^{\frac{1}{p(x,t)}%
}\right\Vert _{p(.,.)} \\
&\leq &C\left\Vert A_{0}\left( x,t,u,\nabla u\right) \vartheta ^{-\frac{1}{%
p(x,t)}}\right\Vert _{q(.,.)}\left\Vert \left\vert \varphi \right\vert
\right\Vert _{W_{0,\vartheta }^{1,p(.,.)}\left( D\times \Omega \right) } \\
&\leq &C\left[ E\left( \dint\limits_{D}\left\vert A_{0}\left( x,t,u,\nabla
u\right) \right\vert ^{q(x,t)}\vartheta ^{-\frac{q(x,t)}{p(x,t)}}dx\right) %
\right] ^{\eta }\left\Vert \left\vert \varphi \right\vert \right\Vert
_{W_{0,\vartheta }^{1,p(.,.)}\left( D\times \Omega \right) }
\end{eqnarray*}%
\begin{eqnarray*}
&\leq &C\left[ E\left( \dint\limits_{D}\left\vert \gamma \left( x,t\right)
+g\left( u\right) \vartheta \left\vert \nabla u\right\vert
^{p(x,t)-1}\right\vert ^{q(x,t)}\vartheta ^{-\frac{q(x,t)}{p(x,t)}}dx\right) %
\right] ^{\eta }\left\Vert \left\vert \varphi \right\vert \right\Vert
_{W_{0,\vartheta }^{1,p(.,.)}\left( D\times \Omega \right) } \\
&\leq &C^{\ast }\left[ E\left( \dint\limits_{D}\left\vert \gamma \left(
x,t\right) \right\vert ^{q(x,t)}\vartheta ^{-\frac{q(x,t)}{p(x,t)}}dx\right)
\right. \\
&&\left. +E\left( \dint\limits_{D}\left\vert g\left( u\right) \right\vert
^{q(x,t)}\left\vert \nabla u\right\vert ^{p(x,t)}\vartheta dx\right) \right]
^{\eta }\left\Vert \left\vert \varphi \right\vert \right\Vert
_{W_{0,\vartheta }^{1,p(.,.)}\left( D\times \Omega \right) } \\
&\leq &C^{\ast }\left[ \rho _{q\left( .,.\right) ,\vartheta ^{\ast }}\left(
\gamma \right) +\max \left\{ \left\Vert g\right\Vert _{\infty
}^{q^{-}},\left\Vert g\right\Vert _{\infty }^{q^{+}}\right\} \rho _{p\left(
.,.\right) ,\vartheta }\left( \nabla u\right) \right] ^{\eta }\left\Vert
\left\vert \varphi \right\vert \right\Vert _{W_{0,\vartheta
}^{1,p(.,.)}\left( D\times \Omega \right) }
\end{eqnarray*}%
where%
\begin{equation*}
\eta =\left\{ 
\begin{array}{c}
\frac{1}{^{q^{-}}},\text{ if }\left\Vert A_{0}\left( x,t,u,\nabla u\right)
\right\Vert _{q(.,.),\vartheta ^{\ast }}\geq 1 \\ 
\frac{1}{^{q^{+}}},\text{ if }\left\Vert A_{0}\left( x,t,u,\nabla u\right)
\right\Vert _{q(.,.),\vartheta ^{\ast }}\leq 1%
\end{array}%
\right. ,
\end{equation*}%
and $\frac{1}{p(.,.)}+\frac{1}{q(.,.)}=1$. Thus, we obtain that $\Gamma _{2}$
is bounded. Therefore, we get that $\Gamma $ is bounded.

\textbf{Step 2.} Now, we will show that the operator $\Gamma $ is coercive.
By $\left( H_{5}\right) $, we have%
\begin{eqnarray*}
\frac{\langle \Gamma _{1}(u),u\rangle }{\left\Vert \left\vert u\right\vert
\right\Vert _{W_{0,\vartheta }^{1,p(.,.)}\left( D\times \Omega \right) }}
&\geq &\frac{E\left( \dint\limits_{D}\alpha \left\vert \nabla u\right\vert
^{p(x,t)}\vartheta (x,t)dx\right) }{\left\Vert \nabla u\right\Vert
_{p(.,.),\vartheta }} \\
&=&\frac{\alpha \rho _{p(.,.),\vartheta }(\nabla u)}{\left\Vert \nabla
u\right\Vert _{p(.,.),\vartheta }}\geq C\left\Vert \nabla u\right\Vert
_{p(.,.),\vartheta }^{r}
\end{eqnarray*}%
for some $r>1.$ On the other hand, since the norm $\left\Vert A_{0}\left(
x,t,u,\nabla u\right) \right\Vert _{q(.,.),\vartheta ^{\ast }}$ is bounded,
then we have 
\begin{eqnarray*}
\left\vert \langle \Gamma _{2}(u),u\rangle \right\vert &\leq &C\left\Vert
A_{0}\left( x,t,u,\nabla u\right) \right\Vert _{q(.,.),\vartheta ^{\ast
}}\left\Vert u\right\Vert _{p(.,.),\vartheta } \\
&\leq &C^{\ast }\left\Vert \left\vert u\right\vert \right\Vert
_{W_{0,\vartheta }^{1,p(.,.)}\left( D\times \Omega \right) }.
\end{eqnarray*}%
This follows that 
\begin{equation*}
\frac{\langle \Gamma (u),u\rangle }{\left\Vert \left\vert u\right\vert
\right\Vert _{W_{0,\vartheta }^{1,p(.,.)}\left( D\times \Omega \right) }}%
\longrightarrow \infty \text{ as }\left\Vert \left\vert u\right\vert
\right\Vert _{W_{0,\vartheta }^{1,p(.,.)}\left( D\times \Omega \right)
}\longrightarrow \infty .
\end{equation*}

\textbf{Step 3. }Now, we will obtain that the operator $\Gamma $ is
pseudo-monotone from $W_{0,\vartheta }^{1,p(.,.)}\left( D\times \Omega
\right) $ to $W_{\vartheta ^{\ast }}^{-1,q(.,.)}\left( D\times \Omega
\right) $. Let $u_{k}\rightharpoonup u$ in $W_{0,\vartheta
}^{1,p(.,.)}\left( D\times \Omega \right) $ and $\limsup\limits_{k%
\longrightarrow \infty }\langle \Gamma \left( u_{k}\right) ,u_{k}-u\rangle
\leq 0$. Since $\Gamma $ is bounded and $u_{k}\rightharpoonup u,$ then we
have 
\begin{equation}
\Gamma \left( u_{k}\right) \rightharpoonup h\text{ in }W_{\vartheta ^{\ast
}}^{-1,q(.,.)}\left( D\times \Omega \right) .  \label{13}
\end{equation}%
By (\ref{13}), we can write that%
\begin{equation}
\limsup\limits_{k\longrightarrow \infty }\langle \Gamma \left( u_{k}\right)
,u_{k}\rangle \leq \langle h,u\rangle .  \label{14}
\end{equation}%
By the growth condition $\left( H_{3}\right) $ and Proposition \ref%
{proposition20}, the sequence $\left( A\left( x,t,u_{k},\nabla u_{k}\right)
\right) _{k\in 
\mathbb{N}
}$ is bounded in $\left( L_{\vartheta ^{\ast }}^{q(.,.)}(D\times \Omega
)\right) ^{d}$ such that 
\begin{equation}
A\left( x,t,u_{k},\nabla u_{k}\right) \rightharpoonup \varphi \text{ in }%
\left( L_{\vartheta ^{\ast }}^{q(.,.)}(D\times \Omega )\right) ^{d}
\label{15}
\end{equation}%
as $k\longrightarrow \infty .$ Similarly, since $\left( A_{0}\left(
x,t,u_{k},\nabla u_{k}\right) \right) _{k}$ is bounded in $L_{\vartheta
^{\ast }}^{q(.,.)}(D\times \Omega )$, then there exists a function $\psi \in
L_{\vartheta ^{\ast }}^{q(.,.)}(D\times \Omega )$ such that 
\begin{equation}
A_{0}\left( x,t,u_{k},\nabla u_{k}\right) \rightharpoonup \psi \text{ in }%
L_{\vartheta ^{\ast }}^{q(.,.)}(D\times \Omega )  \label{16}
\end{equation}%
as $k\longrightarrow \infty .$ For all $v\in W_{0,\vartheta
}^{1,p(.,.)}\left( D\times \Omega \right) ,$ we have%
\begin{eqnarray}
\langle h,v\rangle &=&\underset{k\longrightarrow \infty }{\lim }\langle
\Gamma \left( u_{k}\right) ,v\rangle  \notag \\
&=&\underset{k\longrightarrow \infty }{\lim }E\left( \dint\limits_{D}A\left(
x,t,u_{k},\nabla u_{k}\right) \nabla vdx\right)  \notag \\
&&+\underset{k\longrightarrow \infty }{\lim }E\left(
\dint\limits_{D}A_{0}\left( x,t,u_{k},\nabla u_{k}\right) vdx\right)  \notag
\\
&=&E\left( \dint\limits_{D}\varphi \nabla vdx\right) +E\left(
\dint\limits_{D}\psi vdx\right) .  \label{17}
\end{eqnarray}%
Since $W_{0,\vartheta }^{1,p(.,.)}\left( D\times \Omega \right) $ is
compactly embedded in $L_{\vartheta }^{p(.,.)}(D\times \Omega )$ by Theorem %
\ref{theorem12}, we get%
\begin{equation}
u_{k}\longrightarrow u\text{ in }L_{\vartheta }^{p(.,.)}(D\times \Omega )%
\text{ and a.e. in }D\times \Omega .  \label{18}
\end{equation}%
By (\ref{16}) and (\ref{18}), we have%
\begin{equation}
E\left( \dint\limits_{D}A_{0}\left( x,t,u_{k},\nabla u_{k}\right)
u_{k}dx\right) \longrightarrow E\left( \dint\limits_{D}\psi udx\right)
\label{19}
\end{equation}%
as $k\longrightarrow \infty .$ On the other hand, if we consider (\ref{14})
and (\ref{17}), then we have 
\begin{eqnarray*}
&&\limsup\limits_{k\longrightarrow \infty }\langle \Gamma \left(
u_{k}\right) ,u_{k}\rangle \\
&=&\limsup\limits_{k\longrightarrow \infty }\left( E\left(
\dint\limits_{D}A\left( x,t,u_{k},\nabla u_{k}\right) \nabla u_{k}dx\right)
+E\left( \dint\limits_{D}A_{0}\left( x,t,u_{k},\nabla u_{k}\right)
u_{k}dx\right) \right) \\
&\leq &E\left( \dint\limits_{D}\varphi \nabla udx\right) +E\left(
\dint\limits_{D}\psi udx\right) .
\end{eqnarray*}%
Hence we obtain%
\begin{equation}
\limsup\limits_{k\longrightarrow \infty }E\left( \dint\limits_{D}A\left(
x,t,u_{k},\nabla u_{k}\right) \nabla u_{k}dx\right) \leq E\left(
\dint\limits_{D}\varphi \nabla udx\right) .  \label{20}
\end{equation}%
Due to $\left( H_{4}\right) $, we have%
\begin{equation*}
E\left( \dint\limits_{D}\left( A\left( x,t,u_{k},\nabla u_{k}\right)
-A\left( x,t,u_{k},\nabla u\right) \right) \left( \nabla u_{k}-\nabla
u\right) dx\right) >0
\end{equation*}%
and%
\begin{eqnarray*}
&&E\left( \dint\limits_{D}A\left( x,t,u_{k},\nabla u_{k}\right) \nabla
u_{k}dx\right) \\
&\geq &-E\left( \dint\limits_{D}A\left( x,t,u_{k},\nabla u\right) \nabla
udx\right) +E\left( \dint\limits_{D}A\left( x,t,u_{k},\nabla u_{k}\right)
\nabla udx\right) \\
&&+E\left( \dint\limits_{D}A\left( x,t,u_{k},\nabla u\right) \nabla
u_{k}dx\right) .
\end{eqnarray*}%
Moreover, by (\ref{15}), we get 
\begin{equation*}
\underset{k\longrightarrow \infty }{\lim \inf }E\left(
\dint\limits_{D}A\left( x,t,u_{k},\nabla u_{k}\right) \nabla u_{k}dx\right)
\geq E\left( \dint\limits_{D}\varphi \nabla udx\right) .
\end{equation*}%
It is obtained that 
\begin{equation}
\underset{k\longrightarrow \infty }{\lim }E\left( \dint\limits_{D}A\left(
x,t,u_{k},\nabla u_{k}\right) \nabla u_{k}dx\right) =E\left(
\dint\limits_{D}\varphi \nabla udx\right)  \label{21}
\end{equation}%
by (\ref{20}). If we consider (\ref{17}), (\ref{19}) and (\ref{21}), then we
have 
\begin{equation*}
\underset{k\longrightarrow \infty }{\lim }\langle \Gamma \left( u_{k}\right)
,u_{k}\rangle =\langle h,u\rangle .
\end{equation*}%
This follows from (\ref{18}) and $\left( H_{3}\right) $ that%
\begin{equation*}
A\left( x,t,u_{k},\nabla u\right) \longrightarrow A\left( x,t,u,\nabla
u\right)
\end{equation*}%
in $\left( L_{\vartheta ^{\ast }}^{q(.,.)}(D\times \Omega )\right) ^{d}$.
Hence, we get%
\begin{equation*}
\underset{k\longrightarrow \infty }{\lim }E\left( \dint\limits_{D}\left(
A\left( x,t,u_{k},\nabla u_{k}\right) -A\left( x,t,u_{k},\nabla u\right)
\right) \left( \nabla u_{k}-\nabla u\right) dx\right) =0.
\end{equation*}%
By Lemma \ref{lemma16}, we obtain 
\begin{equation*}
u_{k}\longrightarrow u\text{ in }W_{0,\vartheta }^{1,p(.,.)}\left( D\times
\Omega \right)
\end{equation*}%
and then $\nabla u_{k}\longrightarrow \nabla u$ a.e. in $D\times \Omega $
for a subsequence denoted by $\left( u_{k}\right) _{k\in 
\mathbb{N}
}$. Since $A$ and $A_{0}$ are Carath\'{e}odory functions, we have%
\begin{eqnarray*}
A\left( x,t,u_{k},\nabla u_{k}\right) &\longrightarrow &A\left(
x,t,u_{k},\nabla u\right) \\
A_{0}\left( x,t,u_{k},\nabla u_{k}\right) &\longrightarrow &A_{0}\left(
x,t,u_{k},\nabla u\right) .
\end{eqnarray*}%
This yields that $h=\Gamma \left( u\right) $ and the operator $\Gamma $ is
pseudo-monotone. Finally, if we consider the Theorem \ref{kritikteo}, then
there exists at least a weak solution of (\ref{1}).
\end{proof}

\end{document}